\documentclass[10pt,a4paper]{article}

\usepackage{a4wide}
\usepackage{amsthm}
\usepackage{amsfonts}
\usepackage{amsmath}
\usepackage[english]{babel}

\usepackage{faktor}
\usepackage{etex}
\usepackage{hyperref}
\usepackage{multimedia}
\usepackage[all]{xy}
\usepackage{amssymb}
\usepackage{graphicx,textpos}
\usepackage[dvipsnames]{xcolor}
\usepackage{helvet}
\usepackage{stmaryrd}
\usepackage{enumerate}
\usepackage{todonotes}
\usepackage{pdfsync}
\usepackage{dsfont}
\usepackage[shortcuts]{extdash}
\usepackage[all]{xy}
  \newdir{ >}{{}*!/-9pt/@{>}}

\newcommand{\R}{\mathbb{R}}
\newcommand{\Q}{\mathbb{Q}}
\newcommand{\Z}{\mathbb{Z}}
\newcommand{\C}{\mathbb{C}}

\newcommand{\E}{\mathbb{E}}

\newcommand{\rst}[1]{\ensuremath{{\mathbin\mid}\raise-.5ex\hbox{$#1$}}}

\DeclareMathOperator{\GL}{GL}

\DeclareMathOperator{\Aut}{Aut}
\DeclareMathOperator{\Aff}{Aff}
\DeclareMathOperator{\aff}{aff}

\DeclareMathOperator{\Pol}{P}

\DeclareMathOperator{\Iso}{Iso}

\author{Jonas Der\'e\thanks{E-mail: {\sf jonas.dere@kuleuven.be}. The author was supported by a postdoctoral fellowship of the Research Foundation -- Flanders (FWO).}
}
\title{\bf NIL-affine crystallographic actions\\ of virtually polycyclic groups}
\date{\today}

\newtheorem{Def}{Definition}[section]
\newtheorem{Ex}[Def]{Example}
\newtheorem{Cor}[Def]{Corollary}
\newtheorem{Thm}[Def]{Theorem}
\newtheorem{Prop}[Def]{Proposition}
\newtheorem{Lem}[Def]{Lemma}
\newtheorem{Rmk}[Def]{Remark}
\newtheorem*{Prop*}{Proposition}
\newtheorem*{Lem*}{Lemma}

\newtheorem*{AFP}{Algebraic Fixed Point Problem}
\newcommand{\suchthat}{\;\ifnum\currentgrouptype=16 \middle\fi|\;}

\newcommand{\compcent}[1]{\vcenter{\hbox{$#1\circ$}}}

\newcommand{\comp}{\mathbin{\mathchoice
		{\compcent\scriptstyle}{\compcent\scriptstyle}
		{\compcent\scriptscriptstyle}{\compcent\scriptscriptstyle}}}
	
	\newcommand\restr[2]{{
			\left.\kern-\nulldelimiterspace 
			#1 
			\vphantom{\big|} 
			\right|_{#2} 
	}}

\newtheorem*{rep@theorem}{\rep@title}
\newcommand{\newreptheorem}[2]{%
	\newenvironment{rep#1}[1]{%
		\def\rep@title{#2 \ref{##1}}%
		\begin{rep@theorem}}%
		{\end{rep@theorem}}}
\makeatother

\newreptheorem{Thm}{Theorem}

\begin{document}

\maketitle

\begin{abstract}
A classical result by K.B.~Lee states that every group morphism between almost crystallographic groups is induced by an affine map on the nilpotent Lie group whereon these groups by definition act. It is the main technique for studying morphisms between virtually nilpotent groups, having important applications in fixed point theory, for example as a tool to compute Nielsen numbers for self-maps on infra-nilmanifolds. In this paper we generalize this result to morphisms between virtually polycyclic groups, which are known to act crystallographically by affine transformations on a nilpotent Lie group. The main method is to study the special class of translation-like actions, which behave well under taking subgroups and restricting the action to invariant right cosets.


\end{abstract}

\section{Introduction}

Every virtually polycyclic group $\Gamma$ admits a NIL-affine crystallographic action \cite{deki02-1}, i.e.~an action by affine transformations on a simply connected and connected (hereinafter called $1$-connected) nilpotent Lie group $N$ which is both properly discontinuous and cocompact. Here, the affine group $\Aff(N)$ of affine transformations is equal to the semidirect product $N \rtimes \Aut(N)$ which acts naturally on $N$ as $$(x,\delta) \cdot n = x \delta(n)$$ for all $n \in N$ and $\left(x, \delta \right) \in \Aff(N)$. This is the natural generalization of the abelian case $N = \R^n$, where $\Aff(\R^n) = \R^n \rtimes \GL(n,\R)$ is the classical group of affine transformations. 

For studying self-maps on infra-solvmanifolds up to homotopy, it is important to find a manageable class of maps which induce every possible group morphism between virtually polycyclic groups. For the strongest result, we need an extra assumption on the NIL-affine crystallographic action of $\Gamma$, namely that the unipotent part of the Zariski closure acts by left translations on $N$. We call the NIL-affine crystallographic action translation-like in that case. The first main result states that such actions exist for every virtually polycyclic group. 

\begin{Thm}
	\label{main2}
	For every virtually polycyclic group $\Gamma$, there exists a translation-like NIL-affine crystallographic action $$\rho: \Gamma \to \Aff(N)$$ with $N$ a $1$-connected nilpotent Lie group.
\end{Thm} 
\noindent The first version of this theorem without the extra condition of acting translation-like is due to K. Dekimpe in \cite{deki02-1} and a new proof was given in \cite{dp15-1} more recently. The techniques we develop in this paper give a new and simplified proof of this result and are moreover important for studying NIL-affine actions of virtually polycyclic groups which are not necessarily cocompact. In Corollary \ref{affineconjugate} we prove that translation-like NIL-affine crystallographic actions are unique up to affine conjugation. 

With this stronger notion of translation-like NIL-affine crystallographic actions, we show that every group morphism between virtually polycyclic groups is induced by an affine map. Here an affine map $\alpha = (x,\delta) \in \aff(N_1,N_2)$ is defined as one of the form $\alpha(n) = \delta(n) x$ for all $n \in N_1$ where $\delta: N_1 \to N_2$ is a continuous group morphism between nilpotent Lie groups and $x \in N_2$. By convention, we let the translation part $x \in N_2$ of an element $\alpha = (x,\delta) \in \aff(N_1,N_2)$ act as a right translation, whereas the translation part of elements in $\Aff(N)$ acts as a left translation. Although the group $\Aff(N)$ forms a subset of $\aff(N,N)$, the context will make clear whether the translation part is a left or a right translation. In Example \ref{affinedifferent} we show that in fact this convention induces the same set of maps $\aff(N_1,N_2)$ as letting $x$ act as a left translation, but it simplifies the statement of certain results. 

\begin{Thm}
	\label{main}
	Let $\varphi: \Gamma_1 \to \Gamma_2$ be a group morphism between virtually polycyclic groups $\Gamma_1, \Gamma_2$  with translation-like NIL-affine crystallographic actions $\rho_i: \Gamma_i \to \Aff(N_i)$ for $i \in \{1,2\}$. There exists an affine map $\alpha \in \aff(N_1,N_2)$ such that $$\rho_2( \varphi(\gamma)) \comp \alpha = \alpha \comp \rho_1(\gamma)$$ for all $\gamma \in \Gamma_1$.
\end{Thm}

\noindent A different formulation of the theorem is that every map between infra-solvmanifolds is homotopic to one induced by an affine map. 

The essence of the proof is to restrict the action of $\Gamma_2$ to the image of $\varphi$ in order to get a surjective group morphism, which allows us to extend the group morphism to the $\Q$-algebraic hull of the virtually polycyclic groups introduced in Definition \ref{hull}. With some extra work, we prove that the linear part of $\alpha$ is uniquely determined by the morphism $\varphi$ and describe the possibilities for the translation part of $\alpha$ in Theorem \ref{fixedpoint}.

For general NIL-affine crystallographic actions, we cannot hope to find an affine map, for which we give a counterexample in Example \ref{onlypoly}. Let $\rho_i: \Gamma \to \Aff(N_i)$ be two NIL-affine crystallographic actions, then we say that $\rho_1$ and $\rho_2$ are polynomially conjugate if there exists a bijection $p: N_1 \to N_2$ such that both $p$ and $p^{-1}$ can be expressed as a polynomial in a (and hence any) Mal'cev basis of $N_i$ and with $$p \comp \rho_1(\gamma) \comp p^{-1} = \rho_2(\gamma)$$ for all $\gamma \in \Gamma$. In \cite{bd00-1} it is shown that all NIL-affine crystallographic actions of $\Gamma$ are polynomially conjugate. So any result about translation-like NIL-affine crystallographic actions has an immediate analogue for general NIL-affine crystallographic actions. 

\begin{Cor}
	Let $\varphi: \Gamma_1 \to \Gamma_2$ be a group morphism between virtually polycyclic groups $\Gamma_1, \Gamma_2$  with NIL-affine crystallographic action $\rho_i: \Gamma_i \to \Aff(N_i)$. There exists a polynomial map $p: N_1 \to N_2$ such that $$\rho_2( \varphi(\gamma)) \comp p = p \comp \rho_1(\gamma)$$ for all $\gamma \in \Gamma_1$.
\end{Cor}
\noindent This corollary is the generalization to all group morphisms of the paper \cite{bd00-1} which only considers automorphisms of a fixed virtually polycyclic group. The techniques we develop are based on that paper as well.

The remainder of this paper is structured as follows. We start by giving some preliminaries about crystallographic actions of virtually polycyclic groups and linear algebraic groups in Section \ref{prel}. Next, Section \ref{sec:trans} gives the first step in the proof, namely by defining translation-like NIL-affine actions and studying their properties. Afterwards we introduce $\Q$-algebraic hulls of virtually polycyclic groups in Section \ref{sec:hull}, which form the good generalization of lattices in nilpotent Lie groups for questions about group morphisms. We apply these results to construct crystallographic actions from a given properly discontinuous action of virtually polycyclic groups in Section \ref{sec:comp}. Finally, Section \ref{sec:main} achieves the main result by combining all the previous work.

\section{Preliminaries}
\label{prel}
There are several different definitions and results we need to introduce before starting the proof of our main result. We give a short overview about virtually polycyclic groups, NIL-affine crystallographic actions and linear algebraic groups. 

\paragraph{Virtually polycyclic groups}

A group $\Gamma$ is called \textbf{virtually polycyclic} (or polycyclic-by-finite) if there exists a subnormal series $$1 = \Gamma_0 \triangleleft \Gamma_1 \triangleleft \ldots \triangleleft \Gamma_k = \Gamma$$ such that for every $i$ the group $\faktor{\Gamma_{i+1}}{\Gamma_i}$ is either finite or infinite cyclic. The number of infinite cyclic groups does not depend on the choice of subnormal series and is called the \textbf{Hirsch length} $h(\Gamma)$ of $\Gamma$. Every virtually polycyclic group has a unique maximal finite normal subgroup.

We recall the following classical result of \cite{miln77-1}, which gives already a first geometric realization of virtually polycyclic groups. 

\begin{Thm}
	\label{milnor}
	Every virtually polycyclic group $\Gamma$ admits a properly discontinuous action on $\R^n$ by affine transformations, i.\@e.\@ there exists a map $\rho: \Gamma \to \Aff(\R^n)$ such that $\Gamma$ acts properly discontinuously. Moreover, we can assume that $\rho(\Gamma) \subseteq \Aff(\Q^n) = \Q^n \rtimes \GL(n,\Q)$. 
\end{Thm}

The last statement of this theorem is not explicitly written in the paper \cite{miln77-1}, but it follows directly from the proof. This result led to the question whether every group which acts properly discontinously on $\R^n$ by affine transformations is virtually polycyclic. A negative answer to this question was produced in \cite{marg87-1} with the construction of a properly discontinuous action of a free group on two generators on $\R^3$.

\paragraph{NIL-affine crystallographic actions}

The notion of NIL-affine crystallographic actions originates from the classical \textbf{Euclidean crystallographic groups}. These groups $\Gamma$ are subgroups of the Euclidean isometries $\Gamma \le \Iso(\E^n)$ which act properly discontinuously and cocompactly on the space $\E^n$. In short, we will call an action with these two properties crystallographic.  
\begin{Def}
	\label{defcryst}
	A continuous action $\Gamma \curvearrowright X$ of a group $\Gamma$ on a topological space $X$ is called \textbf{crystallographic} if it is both properly discontinuous and cocompact.
\end{Def}
From simplicity of future statements, we will consider \textbf{Euclidean crystallographic actions} $\rho: \Gamma \to \Iso(\E^n)$ instead of considering only crystallographic subgroups of $\Iso(\E^n)$. Due to the classical results of Bieberbach \cite{bieb11-1,bieb12-1} and Zassenhaus \cite{zass48-1}, the groups admitting an Euclidean crystallographic action are exactly the finitely generated virtually abelian groups. This gives a nice geometric interpretation of these groups and moreover, given an Euclidean crystallographic action $\rho: \Gamma \to \Aff(\R^n)$, it is shown that every automorphism $\varphi: \Gamma \to \Gamma$ is induced by an affine transformation $\alpha \in \Aff(\R^n)$ in the same sense as in Theorem \ref{main}.

The natural question afterwards was to allow for a bigger transformation group and try to find similar algebraic characterizations. The first generalization due to L.~Auslander was to take isometric actions on $1$-connected nilpotent Lie groups, so considering actions of the form $\rho: \Gamma \to N \rtimes C$ where $C \le \Aut(N)$ is a maximal compact subgroup. If the action is crystallographic, then the image of the map $\rho$ is called an \textbf{almost crystallographic group} and the generalized Bieberbach theorems \cite{deki96-1} state that the groups admitting such an action are exactly the finitely generated virtually nilpotent groups. In this case, the projection of $\Gamma$ on the second component is a finite subgroup of $\Aut(N)$, which we will use in the examples lateron. Again, every group morphism is induced by an affine transformation on $N$ by the generalized second Bieberbach theorem \cite{deki96-1}.

Around the same time, a second generalization appeared, namely by studying crystallographic actions $\rho: \Gamma \to \Aff(\R^n)$ by affine transformations. In the beginning, it was believed that this would characterize the virtually polycyclic groups, but Y.~Benoist gave an example of a nilpotent group which does not have an affine crystallographic action \cite{beno95-1}. The other direction, namely whether every group which has an affine crystallographic action is virtually polycyclic is still widely open and is known as the \textbf{Auslander Conjecture}, see \cite{abel01-1}.

Because of the example of Y.~Benoist, it was necessary to consider more general actions in order to include all virtually polycyclic groups. The correct generalization is to work not only with affine transformation of $\R^n$ but of a general $1$-connected nilpotent Lie group $N$. The affine group $\Aff(N)$ is defined as $N \rtimes \Aut(N)$, where the action of $\Aff(N)$ on $N$ is given by $$(n,\delta) \cdot x = n \delta(x)$$ for all $x \in N$ and $\left(n, \delta \right) \in \Aff(N)$. An action of the form $\rho: \Gamma \to \Aff(N)$ on $N$ is called a \textbf{NIL-affine action} and it was shown by K. Dekimpe that every virtually polycyclic group does admit a NIL-affine crystallographic action \cite{deki02-1}. Whether every group which has a NIL-affine crystallographic action is automatically virtually polycyclic is a generalized version of the Auslander conjecture and hence also widely open. 
 
Let $N_1, N_2$ be $1$-connected nilpotent Lie groups. We define the set $\aff(N_1,N_2)$ as the elements $\alpha = (x,\delta)$ with $x \in N_2$ and $\delta: N_1 \to N_2$ a continuous morphism between Lie groups. Every element $\alpha \in \aff(N_1,N_2)$ uniquely corresponds to a map $\alpha: N_1 \to N_2$ given by $\alpha(n) = \delta(n) x$. We emphasize here that the translation part $x \in N_2$ is a right translation, contrary to the convention in some other texts, since this simplifies certain results and proofs, in particular the proof of Theorem \ref{fixedpoint}. The set of induced maps $N_1 \to N_2$ does not depend on this convention by the following remark.

\begin{Rmk}
	\label{affinedifferent}
	Let $N_1, N_2$ be $1$-connected nilpotent Lie groups with $x \in N_2$ and $\delta: N_1 \to N_2$ a continuous group morphism. The map $\alpha: N_1 \to N_2$ given by $\alpha(n) = x \delta(n) $ is an element of $\aff(N_1,N_2)$. Indeed, we can write $$\alpha(n) = x \delta(n) = x \delta(n) x^{-1} x  = \tilde{\delta}(n) x$$ where $\tilde{\delta}: N_1 \to N_2$ is the composition of $\delta$ with conjugation by $x$, i.e.~$\tilde{\delta}(n) = x \delta(n) x^{-1}$. We conclude that $\alpha \in \aff(N_1,N_2)$ and vice versa every element in $\aff(N_1,N_2)$ can be written as the composition of an automorphism and a left translation.
\end{Rmk}

The action of $\Aff(N)$ on $N$ on the other hand is always such that $N$ acts by left translations on itself. Since we use $\Aff(N)$ and $\aff(N_1,N_2)$ in two very distinct ways, there will be no confusion in the text which convention is used.

\paragraph{Linear algebraic groups}

For every subfield $K \subseteq \C$ of the complex numbers, we call a subgroup $G \le \GL(n,\C)$ a \textbf{linear algebraic $K$-group} if it is $K$-closed, i.e.~$G$ is the zero set of a finite number of polynomials with coefficients in $K$. For our purposes, $K$ is usually $\Q$ or $\R$. The subgroup $G(K) = G \cap \GL(n,K)$ of \textbf{$K$-rational points} in $G$ is a Zariski-dense subset of $G$, see \cite[Section 34.4]{hump81-1} and we often abuse notations by just working with the group $G(K)$ instead of the full group $G$. In the case $K = \R$ we call the group of $\R$-rational or real points $G(\R)$ a \textbf{real algebraic group}. The connected component for the Zariski topology of the identity element in $G$ is denoted as $G^0$ and this is a normal subgroup of finite index in $G$. A group morphism between two linear algebraic $K$-groups is said to be a \textbf{$K$-morphism} (or defined over $K$) if the coordinate functions are given by polynomials over the field $K$. 

The \textbf{unipotent radical} $U(G)$ is defined as the set of unipotent elements of the maximal closed normal solvable subgroup of $G$. It is always a closed subgroup which is defined over the field $K$. Most linear algebraic groups $G$ we consider are virtually solvable and in this case the unipotent radical $U(G)$ is just given by the set of unipotent elements in $G$. Note that a $K$-morphism $\varphi: G \to G^\prime$ maps unipotent elements to unipotent elements and in particular $\varphi(U(G)) = U(\varphi(G))$. A torus is a linear algebraic group which is isomorphic to a closed subgroup of diagonal matrices $D(n,\C)$. If the unipotent radical of $G$ is trival the group $G$ is called reductive, which in the virtually solvable case means that the group is virtually a torus. A \textbf{Levi subgroup} is a (reductive) linear algebraic subgroup $L \le G$ such that $G$ is the semi-direct product of $L$ and $U(G)$. Since $K$ has characteristic $0$, \cite[Chapter VIII]{hoch81-1} shows that there always exist Levi subgroups defined over $K$ and they are unique up to conjugation by $K$-rational elements. Moreover, every reductive linear algebraic subgroup of $G$ is contained in a Levi subgroup. For more details about these groups, we refer to \cite{bore91-1,hoch81-1,hump81-1}. 

It will be convenient in this paper to not only study NIL-affine actions, but more generally actions by polynomial diffeomorphisms as in \cite{bd00-1}. We denote by $\Pol(\R^n)$ the group of all polynomial bijections of $\R^n$ and by $\Pol^d(\R^n)$ the subset of polynomial bijections $p: \R^n \to \R^n$ in $\Pol(\R^n)$ with both the degree of $p$ and $p^{-1}$ bounded by $d$. It is shown in \cite{bd00-1} that $\Pol^d(\R^n)$ can be equipped with a Zariski topology for every $d \geq 1$ such that the inclusions $\Pol^d(\R^n) \subseteq \Pol^{d+k}(\R^n)$ are closed immersions for all $k \geq 0$. The Zariski closure of any subgroup $\Gamma \subseteq \Pol^d(\R^n)$ is hence a real algebraic group which does not depend on $d$. An action $\rho: \Gamma \to \Pol(\R^n)$ is called \textbf{of bounded degree} if $\rho(\Gamma) \subseteq \Pol^d(\R^n)$ for some $d \geq 1$. Similarly, one can also define the subset $\Pol^d(\Q^n)$ of polynomial diffeomorphisms with coefficients in $\Q$, which is exactly the subset of rational points in $\Pol^d(\R^n)$. 

If $N$ is a $1$-connected nilpotent Lie group, then the group of affine transformations $\Aff(N)$ can be considered as a closed subgroup of $\Pol^c(\R^n)$ where $c$ is the nilpotency class of $N$ and $n$ the dimension of $N$. If $N$ is actually defined over $\Q$ as a linear algebraic group, then we write the subgroup of rational points as $N^\Q = N(\Q)$ and in this case $\Aff(N^\Q) = N^\Q \rtimes \Aut(N^\Q)$ can be considered as a subgroup of $\Pol^c(\Q^n)$ where $c$ is the nilpotency class of $N^\Q $ and $n$ the dimension. In the special case of $N = \R^n$ we write $\Aff(\Q^n) = \Q^n \rtimes \GL(n,\Q)$ as in Theorem \ref{milnor}. \label{NQ} 

\section{Translation-like NIL-affine actions}
\label{sec:trans}

The main tool in this paper is the introduction a specific type of NIL-affine actions, where the unipotent part acts by left translations. In this section we motivate the definition and show that this property behaves well under taking subgroups and restricting to invariant right cosets.

The first geometric description of virtually polycyclic groups was given in \cite{di96-1}, where the authors showed that every such a group admits a polynomial crystallographic action. In \cite{deki02-1} this result is improved by showing that every virtually polycyclic group admits a NIL-affine crystallographic action. Although these actions already have nice properties, the following example due to \cite{fg83-1} shows that even for nilpotent groups they can behave very differently from what we expect.

\begin{Ex}
	\label{polyZ2}
	\label{notaffine}
	Take $\Gamma = \Z^2$ and consider the affine action 
	\begin{align*}
		\rho: \Gamma &\to \Aff(\R^2) = \R^2 \rtimes \GL(2,\R) \\
		(s,t) & \mapsto \left( \begin{pmatrix} s + \frac{t^2}{2}\\ t \end{pmatrix}, \begin{pmatrix} 1 & t \\ 0 & 1 \end{pmatrix} \right).
	\end{align*}
	It is easy to check that $\rho$ is an injective group morphism and that the action of $\Gamma$ on $\R^2$ via $\rho$ is a NIL-affine crystallographic action. This example was also described in \cite{fg83-1}.
\end{Ex}

So even though the group $\Z^2$ is nilpotent (even abelian), there are affine crystallographic actions which are very different from the standard action by left translation. The following definition tries to catch the behavior where a group acts as much as possible by left translations. 

\begin{Def}
	\label{translationlike}
	Let $\rho: \Gamma \to \Aff(N)$ be a NIL-affine action of a virtually polycyclic group. We call the action $\rho$ \textbf{translation-like} if there exists a real linear algebraic group $G \le \Aff(N)$ with $\rho(\Gamma) \le G$ and such that the unipotent radical $U(G)$ acts by left translations on $N$.
\end{Def}
\noindent Note that for this definition it is crucial that we work with real algebraic groups and hence only consider the real points of the linear algebgraic group. 

From now on, we will consider $N \le \Aff(N)$ as the subgroup of left translations. We will also write $L_m \in \Aff(N)$ for $$L_m: N \to N: n \mapsto m n$$ the left translation associated with $m$. The multiplication in $\Aff(N)$ satisfies $ \delta \comp L_n = L_{\delta(n)} \comp \delta $ for every $\delta \in \Aut(N)$ and $n \in N$. We write $\overline{\rho(\Gamma)}$ for the real Zariski closure of the group $\rho(\Gamma)$ in the real algebraic group $\Aff(N)$. An equivalent formulation of translation-like NIL-affine actions is stating that $U \left(\overline{\rho(\Gamma)}\right)$ acts by left translations, in particular $U \left(\overline{\rho(\Gamma)}\right) \le N$ is a subgroup in this case. Theorem \ref{subaction} will show that a properly discontinuous translation-like NIL-affine action is cocompact if and only if $U \left(\overline{\rho(\Gamma)}\right) = N$. 

The class of translation-like actions behaves well under affine conjugation.
\begin{Ex}
	\label{exconjugate}
Let $\rho:\Gamma \to \Aff(N)$ be a translation-like action. For every $\alpha \in \Aff(N)$, we can define a new NIL-affine action $\rho_0: \Gamma \to \Aff(N)$ given by 
\begin{align*}
\rho_0(\gamma) = \alpha \rho(\gamma) \alpha^{-1}.
\end{align*}
We claim that $\rho_0$ is again translation-like. A computation shows that if $L_{m}$ is a left translation with $m \in N$ and $\alpha \in \Aff(N)$ is of the form $\alpha = (n,\delta)$, then  $$\alpha L_m \alpha^{-1} = L_{n \delta(m) n^{-1}}.$$ The conjugation map $$\Aff(N) \to \Aff(N): \beta \mapsto \alpha \beta \alpha^{-1}$$ is a morphism of linear algebraic groups and hence $$U\left( \overline{\rho_0(\Gamma)} \right) = \alpha \hspace{0.5mm} U \left( \overline{\rho(\Gamma)} \right) \alpha^{-1}.$$ Combined this implies that $\rho_0$ is indeed translation-like.
\end{Ex}

\noindent Corollary \ref{affineconjugate} will show that translation-like NIL-affine crystallographic actions are unique up to affine conjugation, so the converse of Example \ref{exconjugate} holds as well.  

Although the natural actions of almost crystallographic groups are translation-like, Example \ref{polyZ2} shows that not every affine crystallographic action is translation-like.

\begin{Ex}
	\label{almostcryst}
	Let $\Gamma$ be an almost crystallographic group, i.e.~it is a subgroup $$\Gamma \le N \rtimes F\le N \rtimes \Aut(N) = \Aff(N)$$ with $F \le \Aut(N)$ a finite subgroup such that $\Gamma$ acts crystallographically on $N$. The action of $\Gamma$ on $N$ is then a translation-like NIL-affine crystallographic action, since $U(N \rtimes F) = N$ and $N$ acts by left translations on itself. In particular, every lattice in a $1$-connected nilpotent Lie group has the natural inclusion as a translation-like NIL-affine crystallographic action. 
	
	If we look at the action $\rho$ of Example \ref{polyZ2}, then we see that the real algebraic closure $G$ of $\rho(\Z^2)$ is equal to 
	\begin{align*}
	G = & \left\{\left( \begin{pmatrix} x + \frac{y^2}{2} \\ y \end{pmatrix}, \begin{pmatrix} 1 & y \\ 0 & 1 \end{pmatrix} \right) \suchthat x, y \in \R \right\} \le \Aff(\R^2).
	\end{align*} 
	So $G$ is a real unipotent subgroup of $\Aff(\R^2)$ but it does not act via left translations, implying that the action $\rho$ is not translation-like.
\end{Ex}
As it follows from the rest of this paper, translation-like NIL-affine crystallographic actions form the good generalization of the well-known concept of almost crystallographic groups for studying group morphisms. For our goals, it is important that translation-like actions behave well under taking subgroups.

\begin{Prop}
Let $\Gamma$ be a virtually polyclyclic group which acts translation-like via the action $\rho: \Gamma \to \Aff(N)$, then every subgroup $\Gamma^\prime \le \Gamma$ acts translation-like as well via the action $$\restr{\rho}{\Gamma^\prime}: \Gamma^\prime \to \Aff(N).$$
\end{Prop}
\begin{proof}
	Note that $\overline{\rho(\Gamma^\prime)} \le \overline{\rho(\Gamma)}$, so every unipotent element of $\overline{\rho(\Gamma^\prime)}$ lies in  $U \left(\overline{\rho(\Gamma)}\right)$. Hence $U \left(\overline{\rho(\Gamma^\prime)}\right) \le U \left(\overline{\rho(\Gamma)}\right)$ because it is defined as the set of all unipotent elements. The statement follows from the definition of translation-like actions.
\end{proof}

Note that this proof strongly uses the fact that $\Gamma$ is virtually polycyclic, since for general linear algebraic groups $G$ the unipotent radical is not the subset of all unipotent elements, but the unipotent elements of the maximal closed normal solvable subgroup.

\smallskip

Another important tool for our main result is studying restrictions of a NIL-affine crystallographic action to invariant right cosets. These restrictions are again NIL-affine crystallographic actions, moreover restrictions of translation-like actions are translation-like. In the following paragraphs we describe how this method works. 

Let $M \le N$ be a connected Lie subgroup of $N$ and consider a right coset $M n_0$ with $n_0 \in N$. Let $\alpha \in \Aff(N)$ be any affine transformation such that $$\alpha \cdot \left(M n_0 \right) = M n_0.$$ Write $ \alpha = (x,\delta)$ with $x \in N$ and $\delta \in \Aut(N)$ and let $p: M  \to M n_0$ be the natural map, which is given by $p(m) = m n_0$. We describe the map $p^{-1} \comp \alpha \comp p: M \to M$. First note that $n_0 \in M n_0$ and hence $\alpha \cdot n_0 = x \delta(n_0) \in M n_0$ or equivalently that $x \delta(n_0) = m_0 n_0$ for some $m_0 \in M$. Now for any $m n_0 \in M n_0$, we have $$\alpha \cdot \left( m n_0 \right) = x \delta(m n_0)  = x \delta(m) x^{-1} x \delta(n_0) =  x \delta(m) x^{-1} m_0 n_0 = m_0 m_0^{-1}  x \delta(m) x^{-1} m_0 n_0 \hspace{1mm} \in \hspace{1mm} M  n_0.$$ This implies that $p^{-1} \comp \alpha \comp p: M \to M$ is given by $$p^{-1} \comp \alpha \comp p(m) = m_0 m_0^{-1} x \delta(m) x^{-1} m_0$$ and thus is an affine map on $M$ with linear part $m \mapsto m_0^{-1} x \delta(m) x^{-1} m_0$ and translation part $m_0 \in M$. In short, we will write this as $r_M(\alpha) = p^{-1} \comp \alpha \comp p \in \Aff(M)$. Note that this map is only defined after fixing the point $n_0$. 

Consider the subgroup $$H = \{\alpha \in \Aff(N) \mid \alpha \cdot \left(M n_0 \right) = M n_0 \} \le \Aff(N)$$ of elements which map the coset $Mn_0$ to itself. We show that $H$ is a closed subgroup of $\Aff(N)$. Let $\alpha_0 \in \Aff(N)$ be the right translation by $n_0$, i.e.~$$\alpha_0(n) = n n_0 = n_0 n_0^{-1} n n_0,$$ which shows that $\alpha_0$ is indeed an element of $\Aff(N)$ identically as in Remark \ref{affinedifferent}. Showing that $H$ is a closed subgroup is equivalent to showing that $\alpha_0^{-1} H \alpha_0$ is closed. Now the subgroup $\alpha_0^{-1} H \alpha_0$ is equal to $$\alpha_0^{-1} H \alpha_0 = \{\alpha \in \Aff(N) \mid \alpha \cdot M = M  \}.$$ It is easy to see that the latter subgroup is exactly equal to the elements $(x,\delta)$ for which $x \in M$ and $\delta(M) = M$, which form a closed subgroup of $\Aff(N)$.

The natural restriction map $r_M: H \to \Aff(M)$ given by $$r_M(\alpha) = p^{-1} \comp \alpha \comp p$$ defines a morphism between real algebraic groups. We want to emphasize that this map depends on the choice of the point $n_0$, but this will not be important for our results. In particular, unipotent elements of $H$ are mapped to unipotent elements of $\Aff(M)$. Assume that $\rho: \Gamma \to \Aff(N)$ is a NIL-affine action. If $\rho(\Gamma) \le H$, we can consider the \textbf{restricted} NIL-affine action $\rho_M: \Gamma \to \Aff(M)$ which is given by $$\rho_M(\gamma) = r_M \left(\rho(\gamma)\right) = p^{-1} \comp \rho(\gamma) \comp p.$$ In this case, we say that $M n_0$ is invariant under the action $\rho$. The following proposition shows that the restriction behaves well for translation-like actions. 
\label{restrictaffine}

\begin{Prop}\label{newaffine}
Let $\rho: \Gamma \to \Aff(N)$ be a translation-like NIL-affine action and $M n_0$ a right coset which is invariant under the action $\rho$, then the restriction $\rho_M: \Gamma \to \Aff(M)$ is translation-like as well.
\end{Prop}

\begin{proof}
Note that $r_M \left(\overline{\rho(\Gamma)} \right)$ is a closed subgroup which contains $\rho_M(\Gamma)$ and hence it suffices to show that the unipotent part of $r_M \left(\overline{\rho(\Gamma)} \right)$ acts by left translations. The unipotent radical of $r_M \left(\overline{\rho(\Gamma)} \right)$ is exactly $r_M \left(U \left(\overline{\rho(\Gamma)}\right) \right)$ since $r_M$ is a morphism of linear algebraic groups. Now the statement is immediate from the definition, since if $\alpha$ is a left translation, meaning that $\alpha(x) = n x$, then $$r_M(\alpha)(m) = p^{-1} \comp \alpha \comp p (m) = p^{-1} (\alpha (m n_0)) = p^{-1} (n m n_0) = n m$$ is again a left translation.
\end{proof}

Note that some of these arguments for constructing the action $\rho_M$ could have been simplified by letting $N \le \Aff(N)$ act by right translations, as we explained in Example \ref{affinedifferent}. But we prefer to maintain the convention that the subgroup $N \le \Aff(N)$ acts by left translations in this paper to avoid possible confusion, especially since left translations play such an important role in Definition \ref{translationlike} above.

\section{$\Q$-algebraic hulls of virtually polycyclic groups}
\label{sec:hull}

In this section, we show that surjective group morphisms between virtually polycyclic groups can be extended to certain linear algebraic groups in which the groups are embedded. This will be the crucial tool for studying general group morphisms between virtually polycyclic groups, by restricting the group morphism to its image to make it surjective. 

The main tool for studying morphisms between lattices $\Gamma$ of a $1$-connected nilpotent Lie group $N$ is the fact that every group morphism $\varphi: \Gamma \to \Gamma$ uniquely extends to a group morphism $\bar{\varphi}: N \to N$ on the Lie group $N$.  In the solvable case, where every lattice is a polycyclic group, this is no longer the case, as can be seen from the following example.

\begin{Ex}
	Consider the solvable Lie group $S = \R^2 \rtimes_\theta \R$ where the action $\theta: \R \to \GL(2,\R)$ is defined as the rotation $$\theta(t) = \begin{pmatrix} \cos(2 \pi t) & -\sin(2 \pi t) \\ \sin(2 \pi t) & \cos(2 \pi t) \end{pmatrix}.$$ Consider the lattice $\Gamma = \Z^2 \rtimes_\theta \Z \subseteq S$, then it follows from the definition of $\theta$ that $\Gamma \cong \Z^3$. The automorphism of $\Gamma$ which permutes the components of $\Z^3$ as the permutation $(1 \hspace{1mm} 2 \hspace{1mm} 3)$ does not extend to an automorphism of $S$.
\end{Ex}

To study group morphisms between virtually polycyclic groups, it turns out to be useful to define the $K$-algebraic hull, as introduced for the first time in \cite{ragh72-1}. Here $K$ is a subfield of the complex numbers, as before.

\begin{Def}
	\label{hull}
Let $\Gamma$ be a virtually polycyclic group. A \textbf{$K$-algebraic hull} of $\Gamma$ is a linear algebraic group $G$ defined over $K$ equipped with an injective group morphism $i: \Gamma \to G$ satisfying the following properties.
	\begin{enumerate}[$(1)$]
		\item The image $i(\Gamma) \le G(K)$ is a Zariski-dense subgroup of $G$;
		\item If $U(G)$ is the unipotent radical of $G$, then $\dim(U(G)) = h(\Gamma)$;
		\item The centralizer of the unipotent radical in $G$ is contained in $U(G)$, or in symbols $$C_{G}(U(G)) = Z(U(G)),$$ where $Z(U(G))$ is the center of $U(G)$.
	\end{enumerate}
\end{Def}
We will usually identify a group $\Gamma$ with its image under the map $i: \Gamma \to G$. Condition $(1)$ of Definition \ref{hull} already implies that the linear algebraic group $G$ is defined over $K$, as it is the Zariski closure of a set of $K$-rational points. For our applications, the fields $K = \Q$ or $K = \R$ will be the most important one. For simplicity, we will denote the Zariski closure of $\Gamma$ as $\overline{\Gamma}$. The following lemma plays a crucial role in the rest of this paper. 
\begin{Lem}
	\label{inequality}
Let $\Gamma$ be a virtually polycyclic subgroup of a linear algebraic group $G$, then $$\dim\left(U\left(\overline{\Gamma}\right) \right) \leq h(\Gamma).$$ 
\end{Lem} 
This is a reformulation of \cite[Lemma 4.36]{ragh72-1}. Condition (2) of Definition \ref{hull} hence states that the unipotent part of $G$ has maximal dimension. In the case of almost crystallographic groups, the notion of $\Q$-algebraic hulls is well-understood and comes from the natural construction of these groups. 

\begin{Ex}
Let $N$ be a $1$-connected nilpotent Lie group and $F \le \Aut(N)$ a finite subgroup of automorphisms. Consider an almost crystallographic group $\Gamma \le N \rtimes F$ as in Example \ref{almostcryst}. We show that if $F$ equals the projection of $\Gamma$ on the second component, then the group $N \rtimes F$ is equal to the real points of the $\Q$-algebraic hull for the group $\Gamma$. 

The first condition is immediate from the condition on the group $F$. The second condition holds since $N$ acts unipotently on itself via left translations and hence $U(N \rtimes F) = N$. For the last condition, consider $(n,f) \in C_{N \rtimes F}(N)$, then $$m = (n,f) m (n,f)^{-1} = n f(m) n^{-1}$$ for all $m \in N$. By looking at the abelianization $\faktor{N}{[N,N]}$, this implies that $f$ is trivial on $\faktor{N}{[N,N]}$. Since $f$ has finite order and in particular is semisimple, we find that $f$ is trivial or hence $(n,f) \in N$. 
\end{Ex}

From the definition it immediately follows how to find a $K$-algebraic hull of a finite index subgroup starting from a $K$-algebraic hull of the whole group.

\begin{Lem}
	\label{finiteindex}
	Let $i: \Gamma \to G$ be the $K$-algebraic hull of a virtually polycyclic group $\Gamma$ and $\Gamma^\prime \le \Gamma$ a subgroup of finite index in $\Gamma$. If we denote by $G^\prime$ the Zariski-closure of $i(\Gamma^\prime)$, then $\restr{i}{\Gamma^\prime}: \Gamma^\prime \to G^\prime$ is a $K$-algebraic hull for $\Gamma^\prime$.
\end{Lem}

\begin{proof}
	This follows immediately from the definition, since $h(\Gamma) = h(\Gamma^\prime)$ and $U(G) = U(G^\prime)$ as $G^\prime$ is a finite index subgroup of $G$. Note that $G^\prime$ is defined over $K$ since it is defined as the closure of the $K$-rational points $i(\Gamma^\prime) \subset G(K)$. \end{proof}

The following theorem shows the importance of these $K$-algebraic hulls for studying group morphisms between virtually polycyclic groups. 

\begin{Thm}
	\label{extensionQhull}
	Let $\varphi: \Gamma_1 \to \Gamma_2$ be a surjective group morphism between virtually polycyclic groups $\Gamma_1$ and $\Gamma_2$ with $K$-algebraic hulls $i: \Gamma_1 \to G_1$ and $j: \Gamma_2 \to G_2$ respectively. Then $\varphi$ uniquely extends to a morphism $\bar{\varphi}: G_1 \to G_2$ defined over $K$.  
\end{Thm}
This is a generalization of \cite[Lemma 4.41.]{ragh72-1} which gives this result for automorphisms only. The proof of \cite[Lemma 4.41.]{ragh72-1} works in this more general situation but for the sake of completeness and to demonstrate where surjectivity plays a role, we recall the proof here.

\begin{proof}
	To simplify notations, we will identify $\Gamma_1$ and $\Gamma_2$ with their images under the maps $i$ and $j$. Consider the subgroup $$ \Gamma =  \left\{(\gamma,\varphi(\gamma)) \in G_1\times G_2 \mid \gamma \in \Gamma_1\right\} \le G_1\times G_2$$ and denote its Zariski closure as $H$. The group $H$ is a linear algebraic group defined over $K$ since both $\Gamma_1\le G_1(K)$ and $\Gamma_2 \le G_2(K)$. Consider the natural projection maps $p_1: H \to G_1, \hspace{1mm} p_2: H \to G_2$ which are again defined over $K$. 
	
	Note that $\Gamma_1\le p_1(H) \le G_1$ and $\Gamma_2 \le p_2(H) \le G_2$, where we use the surjectivity of $\varphi$ for the second statement. Since $\Gamma_1$ and $\Gamma_2$ are Zariski-dense in $G_1$ and $G_2$ respectively and $p_1(H), p_2(H)$ are closed subgroups, we find that both $p_1$ and $p_2$ are surjective morphisms. In particular, $$p_i(U(H)) = U(p_i(H)) = U(G_i)$$ for $i \in \{1,2\}$. From Lemma \ref{inequality} we get that $$\dim(U(H)) \leq h(\Gamma) = h(\Gamma_1) = \dim (U(G_1)) = \dim(p_1(U(H))) \leq \dim(U(H))$$ and hence that $p_1$ induces an isomorphism between $U(H)$ and $U(G_1) = p_1(U(H))$. 
	
	Now we continue exactly as in the proof of \cite{ragh72-1}. We claim that the kernel of $p_1$ is trivial. Indeed, assume that $p_1(h) = e$ for $h \in H$. The element $h$ is necessarily semisimple because $$\restr{p_1}{U(H)}: U(H) \to U(G_1)$$ is an isomorphism. For every $u \in U(H)$, we have that $p_1(h u h^{-1}) = p_1(u)$. Since both $u, h u h^{-1} \in U(H)$ and again using the property that $p_1$ induces an isomorphism between $U(H)$ and $U(G)$, we get that $h u h^{-1} = u$ for all $u \in U(H)$. By applying the map $p_2$, we find that $p_2(h) \in C_{G_2}(U(G_2))$ and therefore $p_2(h) \in U(G_2)$ because $G_2$ is a $K$-algebraic hull of $\Gamma_2$. Since $p_2(h)$ is semisimple, we find that $p_2(h) = e$ and together with $p_1(h) = e$ this implies $h = e$.
	
	Since the kernel of $p_1$ is trivial, $p_1$ forms an isomorphism defined over $K$. The morphism $\bar{\varphi} = p_2 \comp p_1^{-1}$ then satisfies the conditions of the theorem. In the same way it is immediate to check that this extension is unique.
\end{proof}

As a standard application, we conclude that a $K$-algebraic hull of a virtually polycyclic group is unique up to $K$-isomorphism, so \textit{$K$}-algebraic hull is well-defined in this case. As another immediate consequence we show that every injective endomorphism on a fixed group $\Gamma$ can be extended as well, which is useful for studying co-hopfian properties of virtually polycyclic groups.

\begin{Cor}
	Let $\Gamma$ be a virtually polycyclic group and $i: \Gamma \to G$ a $K$-algebraic hull of $\Gamma$. Every injective group morphism $\varphi: \Gamma \to \Gamma$ uniquely extends to a morphism $\bar{\varphi}: G \to G$ defined over $K$.
\end{Cor}

\begin{proof}
The image $\varphi(\Gamma)$ is a subgroup of finite index in $\Gamma$ and hence $\overline{\varphi(\Gamma)}$ is the $K$-algebraic hull of $\varphi(\Gamma)$ by Lemma \ref{finiteindex}. Since $\varphi: \Gamma \to \varphi(\Gamma)$ is surjective we get a morphism $\bar{\varphi}: G \to \overline{\varphi(\Gamma)}$ defined over $K$ and by composing by the natural inclusion $\overline{\varphi(\Gamma)} \hookrightarrow G$ the statement follows.
\end{proof}

The following examples demonstrate that a general group morphism does not extend to a morphism between $\Q$-algebraic hulls, hence the results stated above are the best possible ones.
\begin{Ex}
	Let $\Gamma = \langle a, b \mid b a b^{-1} = a^{-1} \rangle $ be the fundamental group of the Klein bottle. This is an almost crystallographic group in $\Iso(\R^2)$, where the map $i: \Gamma \to \Iso(\R^2)$ is given by 
	\begin{align*}
	i(a) &= \left( \begin{pmatrix} 1 \\ 0 \end{pmatrix} , \begin{pmatrix} 1 & 0 \\ 0 & 1 \end{pmatrix} \right) \\
	i(b) &= \left( \begin{pmatrix} 0 \\ \frac{1}{2} \end{pmatrix} , \begin{pmatrix} -1 & 0 \\ 0 & 1 \end{pmatrix} \right) .
	\end{align*}
	 Consider the injective group morphism $\varphi: \Z \to \Gamma$ which maps $1$ to $b$. The $\Q$-algebraic hulls of $\Z$ and $\Gamma$ are equal to $\C$ and $\C^2 \rtimes \Z_2$ where $\Z_2$ acts as $$\begin{pmatrix} - 1 & 0 \\ 0 & 1 \end{pmatrix}$$ on $\C^2$. There does not exist a morphism $\psi: \C \to \C^2 \rtimes \Z_2$ between linear algebraic groups which is an extension of $\varphi$ since $1$ is unipotent in $\C$ but $i(b)$ is not unipotent in $\C^2 \rtimes \Z_2$. 
\end{Ex}

Given the $\Q$-algebraic hull for virtually polycyclic groups, the results of this section allow us to study certain group morphisms between the groups. At the moment, it is not clear which virtually polycyclic groups have a $\Q$-algebraic hull, but we give a full answer for this question at the end of the next section.

\section{Constructing cocompact actions}
\label{sec:comp}
Theorem \ref{extensionQhull} of the previous section only gives us extensions to the $\Q$-algebraic hull for surjective group morphisms. The key ingredient to prove Theorem \ref{main} for a group morphism $\varphi: \Gamma_1 \to \Gamma_2$ is to consider the action of the image $\varphi(\Gamma_1) \subset \Gamma_2$ on $N_2$. This action might not be cocompact anymore, but in this section we give a way to construct a new crystallographic action from just a properly discontinuous action of a virtually polycyclic group. 

To study when an action is cocompact, the following fact from algebraic topology is important.
\begin{Prop}
	\label{cocompact}
	Let $\Gamma$ be a virtually polycyclic group which acts properly discontinuously on $\R^k$, then the Hirsch length $h(\Gamma)$ satisfies $h(\Gamma) \leq k$ with equality if and only if the action is cocompact.
\end{Prop}

\begin{proof}
The fact that $h(\Gamma) \leq k$ and that if the action is cocompact, then $h(\Gamma) = k$ are well-known. So the only statement left to check is that if $h(\Gamma) = k$, then the action is cocompact. By taking a finite index subgroup of $\Gamma$ if necessary, we can assume that the action of $\Gamma$ on $\R^k$ is free. The quotient space $M_1 =  \faktor{\R^{k}}{\Gamma}$ is then an aspherical manifold of dimension $k$ and we show that it is necessarily compact. 

The group $\Gamma$ is the fundamental group of a compact aspherical manifold $M_2$ of dimension $k$. By taking a subgroup of index $2$ in $\Gamma$ we can moreover assume that $M_2$ is an orientable manifold. Since the manifolds $M_1$ and $M_2$ have the same homotopy type by \cite[Theorem 1B.8]{hatc02-1}, there is an isomorphism between the $k$-th homology groups, so $H_k(M_1) \cong H_k(M_2)$. Since $M_2$ is orientable, we get that $H_k(M_2) = \Z$. If $M_1$ would not be compact, then $H_k(M_1) = 0$ by \cite[Proposition 3.29]{hatc02-1}, which is a contradiction. 
\end{proof}

In the rest of this section, we fix a properly discontinuous action $\rho: \Gamma \to \Pol(\R^n)$ of a virtually polycyclic group on $\R^n$ such that $\rho(\Gamma) \le \Pol^d(\R^n)$ for some $d \geq 1$. We write $$\gamma \cdot x = \rho(\gamma)(x)$$ for $\gamma \in \Gamma$ and $x \in \R^n$. Let $G = \overline{\rho(\Gamma)} \le \Pol^d(\R^n)$ be the real Zariski closure of $\rho(\Gamma)$ in $\Pol^d(\R^n)$ and define the subgroups $L(\Gamma), \hspace{0,5mm} U(\Gamma) \le G$ as respectively a Levi subgroup and the unipotent radical of $G$. Note that the group $L(\Gamma)$ is not unique, but this has no further consequences for our purposes. In particular $G = U(\Gamma) L(\Gamma)$ with a natural morphism $L(\Gamma) \to \Aut(U(\Gamma))$ of real algebraic groups given by conjugation. Hence there exists a NIL-affine action of $G$ on $U(\Gamma)$ via $$(v l) \cdot u = v l u l^{-1}$$ for all $u, v \in U(\Gamma), \hspace{1mm} l \in L(\Gamma)$, leading to a translation-like NIL-affine action of $\Gamma$ on $U(\Gamma)$. Multiplication in $G$ of $u_1 l_1$ and $u_2 l_2$ with $u_i \in U(\Gamma)$ and $l_i \in L(\Gamma)$ is given by \begin{align} \label{multi} u_1 l_1 u_2 l_2 = \underbrace{u_1 l_1 u_2 l_1^{-1}}_{\in U(\Gamma)} \underbrace{l_1 l_2}_{\in L(\Gamma)}.\end{align} 

The following theorem is a crucial tool for our results. It gives us a crystallographic action starting from only a properly discontinuous action.

\begin{Thm}
	\label{subaction}
	Let $\Gamma$ be a virtually polycyclic group and $\rho: \Gamma \to \Pol(\R^n)$ a properly discontinuous action on $\R^n$ of bounded degree. Assume that $x_0$ is a fixed point of the Levi group $L(\Gamma)$, then the subgroup $U(\Gamma)$ has dimension $h(\Gamma)$ and the group $\Gamma$ acts crystallographically on the closed subset $U(\Gamma) \cdot x_0$. The natural map $$p: U(\Gamma) \to U(\Gamma) \cdot x_0$$ is a homeomorphism which forms a conjugation between the translation-like NIL-affine action of $\Gamma$ on $U(\Gamma)$ and the given action of $\Gamma$ on $U(\Gamma) \cdot x_0$. 
\end{Thm}

This proposition should be seen as a generalization of \cite[Proposition 2.]{bd00-1} and \cite[Lemma 2.9.]{dp15-1} which consider the situation of crystallographic actions. From \cite[Lemma 2.]{bd00-1} it follows that the orbit $U(\Gamma) \cdot x$ is a closed subset of $\R^n$. In the special case of crystallographic actions it holds that $U(\Gamma) \cdot x = \R^n$ since $U(\Gamma) \cdot x$ is a closed subset of dimension $n$ in $\R^n$ and hence equal to the whole space.

\begin{proof}
 From Lemma \ref{inequality} it follows that $\dim(U(\Gamma))  \leq h(\Gamma)$. Consider the subset $$X = U(\Gamma) \cdot x_0 = U(\Gamma) L(\Gamma) \cdot x_0 = G \cdot x_0$$ which is homeomorphic to $\R^k$ for some $0 \leq k \leq \dim(U(\Gamma))$ by \cite[Lemma 2.]{bd00-1}. Since $X$ is the orbit of $x_0$ under the action of the real algebraic closure $G = \overline{\Gamma}$, it is invariant under $\Gamma$ as well, or thus $\Gamma$ indeed acts on the subset $X$. 
 
 Since $\Gamma$ acts properly discontinuously on $\R^n$ and a fortiori on $X$, this implies by Proposition \ref{cocompact} that $k \geq  h(\Gamma)$ and by the previous inequalities we conclude that $k = h(\Gamma)$. In particular, the action of $\Gamma$ on $X$ is cocompact by again applying Proposition \ref{cocompact}. Moreover, the kernel of the action of $U(G)$ on $X$ must have dimension $0$, implying that the action is free. The map $p: U(\Gamma) \to X: u \mapsto p(u) = u \cdot x_0 $ is hence a homeomorphism between $U(\Gamma)$ and $X$. 
 
 By using the multiplication given by equation (\ref{multi}) above, we get that the polynomial diffeomorphism $p: U(\Gamma) \to X$ satisfies $$(v l) \cdot p(u) = (v l) \cdot \left(u \cdot x_0 \right) = \left( v l u \right) \cdot x_0 =  \left(v l u l^{-1}  l \right) \cdot x_0 = \left( (v l) \cdot u \right) \cdot x_0 = p((v l) \cdot u)$$ and so the map $p$ conjugates between the action $\rho$ and an affine action of $\Gamma$ on $U(\Gamma)$. Note that $U(\Gamma)$ acts by left translations for this action, so the action of $\Gamma$ on $X$ is indeed a translation-like NIL-affine crystallographic action.
\end{proof}

In the case where $\rho$ is a translation-like NIL-affine action, the subset $U(\Gamma) \cdot x_0$ is equal to a right coset $U(\Gamma) x_0$. Hence under this assumption the restriction of the action $\rho_{U(\Gamma)}: \Gamma \to \Aff \left(U(\Gamma)\right)$ is again a translation-like NIL-affine action, see Proposition \ref{newaffine}, which is moreover crystallographic. This will be the main tool for achieving the main result of this paper.

To apply Theorem \ref{subaction} we are mainly interested in the case where $L(\Gamma)$ has fixed points in $\R^n$. It seems like the following problem is still open in full generality and is known as the \textit{Algebraic Fixed Point Problem}, see for example \cite{kraf95-1,kraf89-1}.

\begin{AFP}
Does every real reductive group in $\Pol^d(\R^n)$ have a fixed point?
\end{AFP}

There are two interesting cases for which such fixed points always exist and these results will suffice for our purposes. The first is when we are working with subgroups of affine transformations, the second for $L(\Gamma)$ when $\Gamma$ acts crystallographically.

\begin{Lem}
	\label{reductivefixed}
	Every reductive subgroup of $\Aff(N)$ has a fixed point $n_0 \in N$. 
\end{Lem}

\begin{proof}
	Fix a Levi subgroup $L \le \Aut(N)$ and note that $L$ is a Levi subgroup of $\Aff(N)$ as well because $N$ is a unipotent normal subgroup of $\Aff(N) = N \rtimes \Aut(N)$. Since every reductive subgroup of $\Aff(N)$ is conjugate to a subgroup of $L$, it suffices to show that $L$ has a fixed point. The statement now follows immediately since $e \in N$ is a fixed point for $L$.
\end{proof}

\begin{Lem}
	\label{fixedpoly}
	Suppose that $\Gamma$ is a virtually polycyclic subgroup of $\Pol(\R^n)$ of bounded degree which acts crystallographically on $\R^n$, then every reductive subgroup of $\overline{\Gamma}$ has a fixed point. 
\end{Lem}
\begin{proof}
Take $G$ the real Zariski closure of $\Gamma$ and let $L(\Gamma)$ be a Levi subgroup of $G$. Since $G$ is virtually solvable, the connected component $L^0$ of $L$ is a torus. From \cite[Lemma 3]{bd00-1} it follows that $L^0$ has a fixed point $x_0 \in \R^n$. Hence similarly as just below the statement of Theorem \ref{subaction} for the group $L^0 U(\Gamma)$ we get that $U(\Gamma) \cdot x_0 = \R^n$ and moreover this action of $U(\Gamma)$ on $\R^n$ is free.

Take the stabilizer $L^\prime$ of the element $x_0$ in the group $G$. The group $L^\prime$ is a closed subgroup of $G$ with $U(\Gamma) \cap L^\prime = \{e\}$. Because the action of $U(\Gamma)$ on $\R^n$ is transitive, it follows that $U(\Gamma) L^\prime = G$ and hence $L^\prime$ is a Levi subgroup of $G$. Since $L^\prime$ by definition has a fixed point and any other reductive subgroup is conjugate to a subgroup of $L^\prime$, the lemma follows.
\end{proof}

\noindent As a consequence of the proof of Lemma \ref{fixedpoly} we get the following observation, which will be useful for characterizing the possibilities for $\alpha$ in Theorem \ref{main}.

\begin{Lem}
	\label{stablevi}
Let $\rho: \Gamma \to \Pol(\R^n)$ be a properly discontinuous action of bounded degree of a virtually polycyclic group $\Gamma$. If $\Gamma$ acts cocompactly on the set $U(\Gamma) \cdot x$ for an element $x \in \R^n$, then the stabilizer of $x$ in the real Zariski closure $\overline{\rho(\Gamma)}$ is a Levi subgroup of $\overline{\rho(\Gamma)}$.
\end{Lem}

Combining the partial results for the Algebraic Fixed Point Problem and Theorem \ref{subaction} shows that every crystallographic action is conjugate to a translation-like NIL-affine crystallographic action.

\begin{Cor}
	\label{everypoly}
	Let $\Gamma$ be a virtually polycyclic group and $\rho: \Gamma \to \Pol(\R^n)$ a crystallographic action of bounded degree. Then $\rho$ is polynomially conjugate to a translation-like NIL-affine crystallographic action. 
\end{Cor}

\begin{proof}
By taking any fixed point $x_0 \in \R^n$ of a Levi subgroup $L(\Gamma)$, it follows that the set $U(\Gamma) \cdot x_0 = \R^n$. The concrete form of the map $p$ of Theorem \ref{subaction} implies that the action $\rho$ is polynomially conjugate to a translation-like NIL-affine crystallographic action.
\end{proof}

As an illustration of Corollary \ref{everypoly} we give the exact polynomial for Example \ref{polyZ2}.

\begin{Ex}
	\label{conjugatepoly}
	Consider the NIL-affine crystallographic action $\rho: \Z^2 \to \Aff(\R^2)$ from Example \ref{polyZ2}. Take the polynomial diffeomorphism $p: \R^2 \to \R^2$ given by $$p(x,y) = \left( x - \frac{y^2}{2}, y \right),$$ with inverse $p^{-1}(x,y) = \left( x + \frac{y^2}{2}, y \right)$. Conjugating the action $\rho$ by $p$ gives us $$p \comp \rho(s,t) \comp p^{-1} = \left( \begin{pmatrix} s \\ t \end{pmatrix}, \begin{pmatrix} 1 & 0 \\ 0 & 1 \end{pmatrix} \right) $$ which corresponds to the standard action of $\Z^2$ by left translation on $\R^2$.
\end{Ex}

Combining the previous with Theorem \ref{milnor}, we find a proof for Theorem \ref{main2}, with even an extra condition on the action of $\rho$. Recall that the group $\Aff(N^\Q)$ in the statement was introduced on page \pageref{NQ}.
\begin{Thm}
	\label{existtl}
	For every virtually polycyclic group $\Gamma$, there exists a translation-like NIL-affine crystallographic action $$\rho: \Gamma \to \Aff(N)$$ with $N$ a simply connected and connected nilpotent Lie group. Moreover, we can assume that $N$ is defined over $\Q$ and that $\rho(\Gamma) \le \Aff(N^\Q)$.
\end{Thm}

\begin{proof}
Let $\Gamma$ be a virtually polycyclic group and $\rho: \Gamma \to \Aff(\R^n)$ with $\rho(\Gamma) \le \Aff(\Q^n)$ a properly discontinuous action, which exists by Theorem \ref{milnor}. The real algebraic closure $G$ of $\rho(\Gamma)$ is defined over $\Q$, thus also the unipotent radical $U(G)$ is defined over  $\Q$. Moreover there exists a Levi subgroup of $G$ which is defined over $\Q$. Take $x_0 \in \Q^n$ a fixed point for the Levi subgroup of $G$, which exists by exactly the same arguments as in Lemma \ref{reductivefixed}. 

The group $\Gamma$ acts crystallographically on $X = U(G) \cdot x_0$ and the latter action is conjugate to a translation-like NIL-affine action on $N = U(G)$ by Theorem \ref{subaction}. The conjugation map of Theorem \ref{subaction} is defined over $\Q$ and hence the rational points of $G$ are conjugate to the rational points of $\Aff(U(G))$, leading to the last condition.
\end{proof}

 As an immediate consequence, we have the existence of a $\Q$-algebraic hull for virtually polycyclic group for which every finite normal subgroup is trivial, by following the proof of \cite{bd00-1}.
\begin{Thm}
	\label{closurehull}
	Let $\Gamma$ be a virtually polycyclic group and $\rho: \Gamma \to \Pol(\R^n)$ a polynomial crystallographic action of bounded degree $d$. The inclusion of $\rho(\Gamma)$ in $\overline{\rho(\Gamma)} \le \Pol^d(\R^n)$ is a $\R$-algebraic hull for $\rho(\Gamma)$. If we assume furthermore that the image $\rho(\Gamma) \le \Pol^d(\Q^n)$, then it is in fact a $\Q$-algebraic hull for $\rho(\Gamma)$. 
\end{Thm}

\begin{proof}
The first condition is immediate and condition $(2)$ follows from Theorem \ref{subaction}, since every Levi subgroup has a fixed point by Lemma \ref{fixedpoly}. For condition $(3)$, write $G = \overline{\rho(\Gamma)} = L(\Gamma) U(\Gamma)$ and assume that $l \in L(\Gamma)$ centralizes $U(\Gamma)$. Take $x \in \R^n$ a fixed point for $L(\Gamma)$, so in particular $l \cdot x = x$. Every other point in $U(\Gamma) \cdot x$ is a fixed point for $l$ as well, because $$l \cdot \left(u \cdot x\right) = (l u) \cdot x = (u l) \cdot x = u \cdot (l \cdot x) = u \cdot x.$$ Since $U(\Gamma) \cdot x = \R^n$ we get that $l$ must be trivial. The condition $\rho(\Gamma) \le \Pol^d(\Q^n)$ implies that  $\overline{\rho(\Gamma)}$ is defined over $\Q$.
\end{proof}

As an easy application of the previous theorem, we get the following.

\begin{Cor}
	\label{existQhull}
	Let $\Gamma$ be a virtually polycyclic group such that every finite normal subgroup is trivial, then $\Gamma$ has a $\Q$-algebraic hull $i: \Gamma \to G$.
\end{Cor}

\begin{proof}
	Take any polynomial crystallographic action $\rho: \Gamma \to \Pol(\Q^n)$ of bounded degree, which always exists, for example the translation-like NIL-affine crystallographic action we constructed in Theorem \ref{existtl}. The kernel of $\rho$ is a finite normal subgroup because the action is properly discontinuous and hence trivial by assumption. Hence $\rho(\Gamma) \cong \Gamma$ and the Zariski closure of $\rho(\Gamma)$ is the $\Q$-algebraic hull of $\Gamma$. 
\end{proof}

Every finite normal subgroup of $\Gamma$ is also a finite normal subgroup of $\overline{\Gamma}$. Since any finite normal subgroup of a $K$-algebraic hull would centralize the unipotent radical, the existence of a $K$-algebraic hull implies that $\Gamma$ has no non-trivial finite normal subgroups. Hence Corollary \ref{existQhull} characterizes the existence of $K$-algebraic hulls for virtually polycyclic groups completely. This result was already given in \cite[Corollary A.9.]{baue04-1}, but by using different methods, closer to the original arguments in \cite{ragh72-1}.

\section{Group morphisms between virtually polycyclic groups}
\label{sec:main}

In this section we discuss the main result of this paper, in which we apply all the previous work. It shows that every group morphism between virtually polycyclic groups is induced by an affine transformation between translation-like NIL-affine crystallographic actions. 

Recall that the set $\aff(N_1,N_2)$ is given by the elements $(x,\delta)$ with $\delta: N_1 \to N_2$ a continuous morphism and $x \in N_2$ which acts as a right translation. By working with right translations the proof of our main result simplifies and especially Theorem \ref{fixedpoint} will take an easier form than its corresponding result in the virtually nilpotent case.

\begin{repThm}{main}
	Let $\varphi: \Gamma_1 \to \Gamma_2$ be a group morphism between virtually polycyclic groups $\Gamma_1, \Gamma_2$  with translation-like NIL-affine crystallographic action $\rho_i: \Gamma_i \to \Aff(N_i)$. There exists an affine transformation $\alpha \in \aff(N_1,N_2)$ such that $$\rho_2( \varphi(\gamma)) \comp \alpha = \alpha \comp \rho_1(\gamma)$$ for all $\gamma \in \Gamma_1$.
\end{repThm}

\begin{proof}
By replacing $\Gamma_i$ by $\rho_i(\Gamma_i)$, we can assume that $\Gamma_i$ is a subgroup of $\Aff(N_i)$ and $\rho_i$ is the inclusion map for $i = 1, 2$. Take $G_1$ the Zariski closure of $\Gamma_1$ in $\Aff(N_1)$ and $G_2$ the Zariski closure of $\varphi(\Gamma_1)$ in $\Aff(N_2)$. Take $x_0$ any fixed point for a Levi subgroup of $G_2$, which exists by Lemma \ref{reductivefixed}. The right coset $U(G_2) x_0$ is invariant under the action of $\varphi(\Gamma_1)$ and hence we can consider the restriction $\left(\rho_2 \right)_{U(G_2)}$ as in the discussion on page \pageref{restrictaffine}. This action is again translation-like by Proposition \ref{newaffine} and moreover cocompact by Theorem \ref{subaction}. It suffices to show the theorem for the restricted action $\left(\rho_2 \right)_{U(G_2)}$, since composing an affine map $\alpha \in \aff\left(N_1,U(G_2)\right)$ by the natural map $$p: U(G_2) \to U(G_2) x_0 \subseteq N_2$$ gives again an affine map $p \comp \alpha \in \aff \left( N_1, N_2 \right)$. So from now on we can assume that the morphism $\varphi$ is surjective. 

A combination of Theorem \ref{extensionQhull} and Theorem \ref{closurehull} implies that there exists a unique extension $\bar{\varphi}: G_1 \to G_2$ of $\varphi: \Gamma_1 \to \Gamma_2$. Let $L_1$ be the stabilizer of $e \in N_1$, which is a Levi subgroup of $G_1$ by Lemma \ref{stablevi}, and take $L_2 = \bar{\varphi}(L_1)$ as Levi subgroup of $G_2$. Let $x \in N_2$ be a fixed point for the Levi subgroup $L_2$, which always exists by again applying Lemma \ref{reductivefixed}.

Note that Theorem \ref{subaction} implies that $U(\Gamma_i) = N_i$ since we work with translation-like cocompact actions. Consider the restriction $\delta: N_1 = U(\Gamma_1) \to N_2 = U(\Gamma_2)$ of $\bar{\varphi}$ to the unipotent radicables. Define the affine map $\alpha = (x,\delta) \in \aff(N_1,N_2)$ where $x$ acts as a right translation. Take an arbitrary $$\gamma = n l \in \Gamma_1 \le G_1 = U(G_1) L_1 = N_1 L_1$$ with $n \in N_1, l \in L_1$ and let $m \in N_1 \le \Aff(N_1)$. 
By using the multiplication as in Equation (\ref{multi}) and the fact that the action of $U(\Gamma_i)$ is by left translation, a computation shows that   

\begin{align*}
\gamma \cdot m = \gamma \cdot \left( m \cdot e \right) = \left( \gamma m \right) \cdot e = \left( n l m l^{-1} l \right) \cdot e = n l m l^{-1} = \gamma m l^{-1}.
\end{align*} 

By using the previous equality we get 
\begin{align*}
\left( \alpha \comp \gamma \right)(m) = \alpha(\gamma m l^{-1}) &= \delta \left( \gamma m l^{-1}\right) \cdot x  \\ &= \bar{\varphi} (\gamma \hspace{0.5mm} m ) \cdot \left( \varphi(l)^{-1} \cdot x \right) \\  &= \varphi(\gamma) \cdot \left( \delta(m) x \right) \\ &= \varphi(\gamma) \cdot  \alpha(m) ,
\end{align*}
where we use that $x$ a fixed point for $\bar{\varphi}(l)^{-1}$ in the fourth equality. We conclude that $\alpha$ satisfies the conditions of the theorem.
\end{proof}

Note that we can even weaken the condition on the action $\rho_2: \Gamma_2 \to \Aff(N_2)$ in the theorem to just being properly discontinuous, since cocompactness is not needed anywhere. We can also interprete this theorem for general polynomial crystallographic actions of bounded degree. 

\begin{Cor}
\label{generalmorph}
Let $\varphi: \Gamma_1 \to \Gamma_2$ be a group morphism between virtually polycyclic groups with polynomial crystallographic actions $\rho_i: \Gamma_i \to \Pol(\R^{n_i})$ of bounded degree. Then there exists a polynomial map $p \in \Pol(\R^{n_1},\R^{n_2})$ such that $$\rho_2( \varphi(\gamma)) \comp p = p \comp \rho_1(\gamma) $$ for all $\gamma \in \Gamma_1$ 
\end{Cor}

\begin{proof}
This follows immediately from a combination of Theorem \ref{main} with Corollary \ref{everypoly}.
\end{proof}

In the general case of polynomial actions in Corollary \ref{generalmorph}, we do need the cocompactness assumption to ensure the existence of a fixed point for a Levi subgroup as in Lemma \ref{fixedpoly}. Since every NIL-affine crystallographic action $\rho: \Gamma \to \Aff(N)$ is a polynomial crystallographic action, Corollary \ref{generalmorph} holds in particular for NIL-affine crystallographic actions as well. As the following examples shows it is not possible to strengthen Corollary \ref{generalmorph} to the existence of an affine map for NIL-affine crystallographic actions. 

\begin{Ex}
	\label{onlypoly}
Let $\rho_1: \Z^2 \to \Aff(\R^2)$ the NIL-affine crystallographic action as in Example \ref{polyZ2} and $\rho_2: \Z^2 \to \Aff(\R^2)$ the standard action by left translations. An easy computation shows that there is no affine map $\alpha \in \aff(\R^2)$ such that $ \rho_2(\gamma) \comp \alpha = \alpha \comp \rho_1(\gamma)$ for all $\gamma \in \Z^2$. There is a polynomial map though, namely the one given in Example \ref{conjugatepoly}.
\end{Ex} 

\noindent Another consequence of Theorem \ref{main} is a strengthened rigidity compared to Theorem \ref{everypoly} for translation-like NIL-affine crystallographic actions. 

\begin{Cor}
	\label{affineconjugate}
	A translation-like NIL-affine crystallographic action is unique up to affine conjugation.  
\end{Cor}

Finally, we consider the question how many affine maps $\alpha \in \aff(N_1,N_2)$ exist that realize a given group morphism between translation-like NIL-affine crystallographic actions. For this we need to deeper analyse the proof of our main theorem above. First we note that the existence of the affine map $\alpha$ immediately leads to a cocompact action of the group $\varphi(\Gamma_1)$.

\begin{Prop}
	\label{imagecryst}
	Let $\varphi: \Gamma_1 \to \Gamma_2$ be a group morphism between virtually polycyclic groups with NIL-affine crystallographic actions $\rho_i: \Gamma_i \to \Aff(N_i)$. Let $\alpha \in \aff(N_1,N_2)$ be as in Theorem \ref{main}, then the group $\varphi(\Gamma_1)$ acts crystallographically on the right coset $\alpha(N_1)$. 
\end{Prop}

\begin{proof}
The condition on $\alpha$ implies that the right coset $\alpha(N_1) = \delta(N_1) x$ is invariant under the action of $\varphi(\Gamma_1)$. Since the action is automatically properly discontinuous, we are left to show that the action is cocompact. Take $K \subset N_1$ a compact subset such that $\Gamma_1 \cdot K = N_1$, then the subset $\alpha(K)$ is a compact subset of $\alpha(N_1)$. We have that $$\varphi(\Gamma_1) \cdot \alpha(K) = \alpha(\Gamma_1 \cdot K) = \alpha(N_1)$$ and hence the action of $\varphi(\Gamma_1)$ is cocompact by definition.
\end{proof}

Next we prove the following result about fixed points.

\begin{Prop}
	\label{fixedpointset}
Let $K \le \Aff(N)$ be a subgroup of affine transformations on a $1$-connected nilpotent Lie group $N$. If we denote the set of fixed points of $K$ as $$V(K) = \{n \in N \suchthat \forall \hspace{1mm} k \in K:  k \cdot n = n  \},$$ then either $V(K) = \emptyset$ or $V(K) = M n_0$ is a right coset with $M$ a connected Lie subgroup of $N$ and $n_0 \in N$.
\end{Prop}

\begin{proof}
We assume that the set of fixed points $V(K) \neq \emptyset$ and we show that it is a right coset of the given form. Take any $n_0 \in V(K)$. Since $$V(K) = \bigcap_{k \in K} V\left(\{k \}\right)$$ is the intersection of fixed point sets of elements in $K$ and the non-empty intersection of right cosets of connected Lie subgroups is again a right coset of a connected Lie subgroup, it suffices to show the theorem for one element $(x,\delta) \in \Aff(N)$ which has a fixed point $n_0$.  

Every element in $N$ is of the form $m n_0$ with $m \in N$. The point $m n_0$ is a fixed point for $(x,\delta)$ if and only if $$m n_0 = (x,\delta) \cdot(m n_0) = x \delta(m) \delta(n_0) = x \delta(m) x^{-1} x \delta(n_0) = x \delta(m) x^{-1} n_0,$$ where we use that $n_0$ is a fixed point of $(x,\delta)$. So $m n_0$ is a fixed point for $(x,\delta)$ if and only if $x \delta(m) x^{-1} = m$. In particular, the fixed point set of $(x,\delta)$ is equal to  $M n_0$ with $M$ the connected Lie subgroup of $N$ given by the latter condition.
\end{proof}

Now we are ready to describe all affine maps realizing a given group morphism. We take the same notations as in Theorem \ref{main}.

\begin{Thm}
	\label{fixedpoint}
For every $\varphi: \Gamma_1 \to \Gamma_2$ there exists a Levi subgroup $L \le \overline{\varphi(\Gamma_1)}$ and a unique $\delta: N_1 \to N_2$ such that every affine map $\alpha \in \aff(N_1,N_2)$ as in Theorem \ref{main} is of the form $(x,\delta)$ with the translation part $x$ a fixed point for $L$.
\end{Thm}

Note that this theorem is the generalization of \cite[Proposition 1.4]{lee95-2} to virtually polycyclic groups, although the notation is slightly different therein since it is stated for affine transformations with left translations. This shows again that the choice of affine maps with right translations can simplify the results, additionally leading to uniqueness for the linear part.

\begin{proof}
Without loss of generality we take $\Gamma_i \le \Aff(N_i)$ as subgroups. Let $G_1, G_2$ be the real Zariski closure of the groups $\Gamma_1$ and $\varphi(\Gamma_1)$ respectively. The first step of the proof is to show that image of $\alpha$ is a subset of some right coset $M n_0$ which is invariant under the group $\varphi(\Gamma_1)$. We then construct the linear map $\delta: N_1 \to M \subseteq N_2$ from the restricted action of $\varphi(\Gamma_1)$ to $M n_0$. Finally we take the inverse image of the Levi subgroup in the restriction to obtain the final result.

First we construct the right coset $M n_0$ on which the group $\varphi(\Gamma_1)$ acts. Let $K$ be the intersection of all Levi subgroups in $G_2$. The group $K$ is a closed, reductive subgroup of $G_2$ and by construction it is normal in $G_2$. This furthermore implies that it centralizes $U(G_2)$, since for every $u \in U(G_2)$ and $k \in K$, we have $[u,k] \in U(G_2) \cap K = \{ e \}$. Lemma \ref{reductivefixed} implies that there is a fixed point $n_0 \in N_2$ for the group $K$ and hence by Lemma \ref{fixedpointset} the fixed point set $V(K) = M n_0$ for some connected Lie subgroup $M \le N_2$. Note that $M n_0$ is invariant under the action of $G_2$ and hence under the action of $\varphi(\Gamma_1)$ as well. 

We now show that if $\alpha$ satisfies the condition of Theorem \ref{main}, then the image $\alpha(N_1)$ always lies in the set $M n_0$. Since $K$ is the intersection of all Levi subgroups, it suffices to show that every point in $\alpha(N_1)$ is a fixed point for a Levi subgroup of $G_2$. Note that $\varphi(\Gamma_1)$ acts crystallographically on $\alpha(N_1)$ by Proposition \ref{imagecryst}. Hence the stabilizer of any element of $\alpha(N_1)$ in $G_2$ is a Levi subgroup by Lemma \ref{stablevi}, as we claimed. 

Consider the restricted action of $\varphi(\Gamma_1)$ to the invariant coset $M n_0$. The group $K$ is a subgroup of the kernel of the restriction map $r_M: G_2 \to \Aff(M)$, hence the intersection of all Levi subgroups in $r_M(G_2) = \overline{r_M(\varphi(\Gamma_1))}$ is trivial. Thus the linear algebraic group $r_M(G_2)$ is the $\R$-algebraic hull of $r_M(\varphi(\Gamma_1))$ by reproducing the proof of Theorem \ref{closurehull}. By Theorem \ref{extensionQhull} there is a unique extension $\bar{\varphi}: G_1 \to r_M(G_2)$. Take $\bar{\varphi}(L_1) = L_2$ as a Levi subgroup of $r_M(G_2)$ and let $\delta: N_1 \to M \subseteq N_2$ be the restriction of $\bar{\varphi}$ to the unipotent radicals, where $U(r_M(G_2)) = r_M(U(G_2)) = U(G_2)$. 

We are now ready to show that every affine transformation $\alpha$ as in Theorem \ref{main} is exactly of the form $\alpha = (x,\delta)$ with $x$ a fixed point for the Levi subgroup $r_M^{-1}(L_2)$ of $G_2$. For the existence part, note that if $x$ is a fixed point of $r_M^{-1}(L_2)$, then $x$ in particular is a fixed point for $K$ and hence lies in the right coset $M n_0$. The existence of $\alpha$ then follows directly from reproducing the proof of Theorem \ref{main} for the restricted action to $M n_0$. 

To show that an arbitrary $\alpha$ is of the desired form,  consider the restricted map $\alpha_M: N_1 \to M n_0$, which is possible since the image of $\alpha$ lies in $M n_0$. Note that the condition of Theorem \ref{main} also holds for the restricted map, i.e. $$\bar{\varphi}(\gamma) \comp \alpha_M = \alpha_M \comp \gamma$$ for all $\gamma \in \Gamma_1$. Since this condition is given by polynomials, it also holds for the real algebraic closure, so $\bar{\varphi}(g) \comp \alpha_M = \alpha_M \comp g$ holds for all $g \in G_1 = \overline{\Gamma_1}$. Since $e$ is a fixed point of $L_1$, we have that $\alpha_M(e) = \alpha(e) = x$ is a fixed point of $\bar{\varphi}(L_1) = L_2$ and hence also of $r_M^{-1}(L_2)$ since $x \in M n_0$. Note that $$\alpha(n) = \alpha(n \cdot e) = \bar{\varphi}(n) (\alpha(e)) = \delta(n) x$$ and thus the theorem follows. 
\end{proof} 

\section{Generalization to solvable Lie groups}

Most of the results in this paper have an analogue for continuous actions of $1$-connected solvable Lie groups $S$. We chose to present the proofs for virtually polycyclic groups before, since these are the most important ones for applications. This section shortly describes how the main results can be generalized to the case of solvable Lie groups, for more details on the definitions we refer to the paper \cite{bd00-1}.

From now on, $S$ denotes a $1$-connected solvable Lie group. The main difference between virtually polycyclic groups and solvable Lie groups are the type of actions we consider. The actions on $\R^n$ we look at are always continuous and polynomial of bounded degree, i.e.~there exists a continuous map $\rho: S \to \Pol^d(\R^n)$ for some $d > 0$. Simple (or free) actions, meaning that for every $x \in \R^n$ and $s \in S$, the equality $s \cdot x = x$ implies  $s = 1$, are the natural counterpart of properly discontinuous actions of discrete groups, whereas transitive actions correspond to cocompact actions. In short we will call the combination of both properties \textbf{simply transitive}, which hence corresponds to crystallographic actions of virtually polycyclic groups as in Definition \ref{defcryst}.

Just as in the proof of \cite[Theorem 1.1.]{miln77-1}, one can show that the group $S$ admits an action $\rho: S \to \Aff(\R^n)$ which is simple. A similar result as Theorem \ref{subaction} holds for every fixed point of a Levi subgroup of $\overline{\rho(S)}$. Hence by exactly the same arguments as for virtually polycyclic groups there exists a simply transitive action $\rho: S \to \Aff(N)$ on some $1$-connected nilpotent Lie group $N$ which is moreover translation-like, meaning that $U \left(\overline{\rho(S)}\right)$ acts by left translations. This gives a new and simplified proof for \cite[Theorem 6.2.]{deki98-1}. 

The notion of $\R$-algebraic hull can be defined just as in Definition \ref{hull} for the Lie group $S$. For every simply transitive action $\rho: S \to \Aff(N)$, the Zariski closure $\overline{\rho(S)}$ will form a $\R$-algebraic hull. Moreover, we have the following result about continuous group morphisms between these actions.

\begin{Thm}
Let $S_1, S_2$ be Lie groups which act simply transitively and translation-like on a nilpotent Lie group $N_i$ via the actions $\rho_i: S_i \to \Aff(N_i)$. For every continuous group morphism $\varphi:S_1 \to S_2$, there exists an affine map $\alpha \in \aff(N_1,N_2)$ such that $$\rho_2( \varphi(s)) \comp \alpha = \alpha \comp \rho_1(s)$$ for all $s \in S_1$.
\end{Thm}

Note that under the conditions of the theorem, the Lie groups $S_i$ are necessarily solvable and $1$-connected. For general NIL-affine actions (or even polynomial actions) which are simply transitive, there always exists a polynomial diffeomorphism between the actions as in Corollary \ref{generalmorph}. Example \ref{onlypoly} again implies that this cannot be improved to affine maps for general simply transitive NIL-affine actions of solvable Lie groups.

\bibliography{ref}

\begin{thebibliography}{10}

\bibitem{abel01-1}
H.~Abels.
\newblock Properly discontinuous groups of affine transformations: A survey.
\newblock {\em Geometriae Dedicata}, 87:309--333, 2001.

\bibitem{baue04-1}
Oliver Baues.
\newblock Infra-solvmanifolds and rigidity of subgroups in solvable linear
  algebraic groups.
\newblock {\em Topology}, 43(4):903--924, 2004.

\bibitem{beno95-1}
Yves Benoist.
\newblock Une nilvari\'et\'e non affine.
\newblock {\em J. Differential Geom.}, 41,:pp. 21--52, 1995.

\bibitem{bd00-1}
Yves Benoist and Karel Dekimpe.
\newblock The uniqueness of polynomial crystallographic actions.
\newblock {\em Math.\ Ann.}, 322((2)):pp. 563--571, 2002.

\bibitem{bieb11-1}
L~Bieberbach.
\newblock {\"{U}}ber die {B}ewegungsgruppen der {E}uklidischen {R}\"aume {I}.
\newblock {\em Math. Ann.}, 70(3):pp. 297--336, 1911.

\bibitem{bieb12-1}
L~Bieberbach.
\newblock {\"{U}}ber die {B}ewegungsgruppen der {E}uklidischen {R}\"aume {II}.
\newblock {\em Math. Ann.}, 72(3):pp. 400--412, 1912.

\bibitem{bore91-1}
Armand Borel.
\newblock {\em Linear algebraic groups}, volume 126 of {\em Graduate Texts in
  Mathematics}.
\newblock Springer-Verlag, second edition, 1991.

\bibitem{deki96-1}
Karel Dekimpe.
\newblock {\em {A}lmost-{B}ieberbach {G}roups: {A}ffine and Polynomial
  Structures}, volume 1639 of {\em Lect. Notes in Math.}
\newblock Springer--Verlag, 1996.

\bibitem{deki98-1}
Karel Dekimpe.
\newblock Semi-simple splittings for solvable lie groups and polynomial
  structures.
\newblock {\em Forum Math.}, 12:pp.\ 77--96, 2000.

\bibitem{deki02-1}
Karel Dekimpe.
\newblock Any virtually polycyclic group admits a nil-affine crystallographic
  action.
\newblock {\em Topology}, 42:pp. 821--832, 2003.

\bibitem{di96-1}
Karel Dekimpe and Paul Igodt.
\newblock Polycyclic-by-finite groups admit a bounded-degree polynomial
  structure.
\newblock {\em Invent. Math.}, 129((1)):pp. 121--140, 1997.

\bibitem{dp15-1}
Karel Dekimpe and Nansen Petrosyan.
\newblock Crystallographic actions on contractible algebraic manifolds.
\newblock {\em Trans. Amer. Math. Soc.}, 367(4):2765--2786, 2015.

\bibitem{fg83-1}
David Fried and William~M. Goldman.
\newblock Three--dimensional affine crystallographic groups.
\newblock {\em Adv. in Math.}, 47(1):pp. 1--49, 1983.

\bibitem{hatc02-1}
Allen Hatcher.
\newblock {\em Algebraic topology}.
\newblock Cambridge University Press, Cambridge, 2002.

\bibitem{hoch81-1}
Gerhard~P. Hochschild.
\newblock {\em Basic theory of algebraic groups and {L}ie algebras}, volume~75
  of {\em Graduate Texts in Mathematics}.
\newblock Springer-Verlag, New York-Berlin, 1981.

\bibitem{hump81-1}
James~E. Humphreys.
\newblock {\em Linear Algebraic Groups}.
\newblock Graduate Texts in Mathematics. Springer-Verlag, New York, 1981.

\bibitem{kraf95-1}
Hanspeter Kraft.
\newblock Challenging problems on affine {$n$}-space.
\newblock {\em Ast\'erisque}, (237):Exp.\ No.\ 802, 5, 295--317, 1996.
\newblock S\'eminaire Bourbaki, Vol. 1994/95.

\bibitem{kraf89-1}
Hanspeter Kraft, Ted Petrie, and Gerald~W. Schwarz, editors.
\newblock {\em Topological methods in algebraic transformation groups},
  volume~80 of {\em Progress in Mathematics}. Birkh\"auser Boston, Inc.,
  Boston, MA, 1989.

\bibitem{lee95-2}
Kyung~Bai Lee.
\newblock Maps on infra-nilmanifolds.
\newblock {\em Pacific J. Math.}, 168,(1):pp. 157--166, 1995.

\bibitem{marg87-1}
Gregory~A. Margulis.
\newblock Complete affine locally flat manifolds with a free fundamental group.
\newblock {\em J. Soviet Math.}, (134):pp. 129--134, 1987.

\bibitem{miln77-1}
John Milnor.
\newblock On fundamental groups of complete affinely flat manifolds.
\newblock {\em Adv. Math.}, 25:pp. 178--187, 1977.

\bibitem{ragh72-1}
M.~S. Raghunathan.
\newblock {\em Discrete {S}ubgroups of {L}ie {G}roups}, volume~68 of {\em
  Ergebnisse der Mathematik und ihrer Grenzgebiete}.
\newblock Springer-{V}erlag, 1972.

\bibitem{zass48-1}
H~Zassenhaus.
\newblock \"uber einen algorithmus zur bestimmung der raumgruppen.
\newblock {\em Commentarii Mathematici Helvetici}, (21):pp. 117--141, 1948.

\end{thebibliography}
\bibliographystyle{plain}

\end{document}